\documentclass[12pt]{amsart}
\usepackage{enumerate, pstcol, amssymb, amsfonts, amsbsy, amsthm, amsmath, amscd, epsfig, latexsym, stmaryrd, epic, eepic, eucal, multicol, fancyhdr, graphicx}
\newtheorem{thm}{Theorem}[section]
\newtheorem{lem}[thm]{Lemma}
\newtheorem{prop}[thm]{Proposition}
\theoremstyle{definition}
\newtheorem{defn}[thm]{Definition}

\newtheorem{rem}[thm]{Remark}
\newtheorem{cor}[thm]{Corollary}

\newcommand{\Pro}{\mathbb P}
\newcommand{\mult}{\text{mult}}

\begin{document}

\title{Polynomials Whose Coefficients Coincide with Their Zeros}
\author{Oksana Bihun and Damiano Fulghesu}
\address[Oksana Bihun]{
Department of Mathematics, University of Colorado, Colorado Springs,
1420 Austin Bluffs Pkwy, Colorado Springs, CO~80918,~USA}\email{ obihun@uccs.edu }
\address[Damiano Fulghesu]{ Department of Mathematics, Minnesota State University Moorhead, 1104 7th Ave. South, Moorhead, MN~56563, USA}\email{ fulghesu@mnstate.edu }

\maketitle
\begin{abstract}

In this paper we consider monic polynomials such that their coefficients coincide with their zeros. These polynomials were first introduced by S.~Ulam. We combine methods of algebraic geometry and dynamical systems to prove several results. We obtain estimates on the number of Ulam polynomials of degree $N$. We provide additional methods to obtain algebraic identities satisfied by the zeros of Ulam polynomials, beyond the straightforward comparison of their zeros and coefficients. To address the question about existence of orthogonal Ulam polynomial sequences, we show that the only Ulam polynomial eigenfunctions of hypergeometric type differential operators are the trivial Ulam polynomials $\{x^N\}_{N=0}^\infty$. We propose a family of solvable $N$-body problems such that their stable equilibria are the zeros of  certain Ulam polynomials.
\end{abstract}



\section{Introduction}

Let $N$ be a positive integer. We identify the space of all monic polynomials of degree $N$ with complex coefficients with $\mathbb C^N$, by describing each monic polynomial in terms of its non-leading coefficients.  Consider the map $\psi^{(N)}: \mathbb C^N \to \mathbb C^N$ defined by
\begin{eqnarray*}
(c_1, c_2, \dots, c_N) & \mapsto & (-s_1^{(N)}, s_2^{(N)}, \dots, (-1)^N s_N^{(N)}),
\end{eqnarray*}
where $s_j^{(N)}$ is the $j^{\text{th}}$ symmetric polynomial in the $N$ variables $c_i$:
\begin{equation}
s^{(N)}_j=s^{(N)}_j(c_1, \ldots, c_N)=\frac{1}{j!} \sum_{n_1, \ldots, n_j=1}^N c_{n_1} \cdots c_{n_j}.
\label{sj}
\end{equation}
The map $\psi^{(N)}$ sends a monic polynomial $x^N+\sum_{m=1}^N c_m x^{N-m}$ with the coefficients $c_1, c_2, \dots, c_N$ into another monic polynomial $\prod_{n=1}^N(x-c_n)$ whose zeros are exactly $c_1, c_2, \dots, c_N$. The map $\psi^{(N)}$ was proposed by Ulam (see~\cite{Ulam}, p.31), hence we will refer to it as the Ulam map. Ulam wrote that ``Many of the statements about algebraic equations are translatable into the elementary properties of this mapping.'' One of the questions posed by Ulam is the identification of nontrivial fixed points of the map $\psi^{(N)}$, the trivial fixed point being $0\in \mathbb{C}^N$. We address this question in the present paper.

Our goal is to investigate the points $\gamma \in \mathbb C^N$ such that $\psi^{(N)}(\gamma) = \gamma$, that is to say the monic polynomials such that their coefficients coincide with their zeros. In the following we will refer to such polynomials as Ulam polynomials. In~\cite{Stein} it is shown that for $N\geq 5$ the Ulam map $\psi^{(N)}$ does not have a fixed point with the property that all its components are real and distinct from zero.  In this paper, we combine methods of algebraic geometry and dynamical systems to address the following problems. We focus on counting the number of the fixed points of $\psi^{(N)}$ and on  proving the existence  of nontrivial complex fixed points of $\psi^{(N)}$.  We provide additional methods to obtain algebraic identities satisfied by the zeros of Ulam polynomials, beyond the straightforward comparison of their zeros and coefficients.  We make progress on the question about existence of orthogonal Ulam polynomial sequences by showing that the only Ulam polynomial eigenfunctions of hypergeometric type differential operators are the trivial Ulam polynomials $\{x^N\}_{N=0}^\infty$, which are eigenfunctions of the differential operator $\alpha x^2 \frac{d^2}{dx^2}-x\frac{d}{dx}$ with the corresponding eigenvalues $\{N(N-1)\alpha-N\}_{N=0}^\infty$. We propose a family of solvable $N$-body problems such that their stable equilibria are the zeros of certain Ulam polynomials. 

Below we provide several useful definitions and theorems that will be used in the subsequent exposition.

\begin{defn} \label{solutions.at.infinity} {\it (Solutions at infinity)}
Let $\alpha_1, \alpha_2, \dots, \alpha_N$ be polynomials in $\mathbb C [c_1, c_2, \dots, c_N]$, respectively, of degrees $d_1, d_2, \dots, d_N$. Consider the projective space $\Pro^N(\mathbb C)$. We introduce the homogeneous coordinates $\left[ C_0 : C_1 : \dots : C_N \right]$ in $\Pro^N(\mathbb C)$ such that $\mathbb C^N$ is the chart corresponding to the coordinates $\displaystyle c_i=\frac{C_i}{C_0}$ for $i=1,2, \dots , N$. We say that the system $\{ \alpha_i =0\}_{i=1, \dots, N}$ has {\it solutions at infinity} if the system $\left \{ \widehat{\alpha}_i = 0 \right \}_{i=1, \dots, N}$ has solutions for $C_0=0$, where
$$
\widehat{\alpha}_i (C_0, C_1, \ldots, C_N):= (C_0)^{d_i} \cdot \alpha_i\left( \frac{C_1}{C_0}, \frac{C_2}{C_0}, \dots, \frac{C_N}{C_0}\right)
$$
is the homogenization of $\alpha_i$.
\end{defn}

\begin{defn}\label{algebraic.sets}
Let $I$ be an ideal in the polynomial ring $\mathbb C [c_1, c_2, \dots, c_N]$. The variety $V(I)$ is defined as the subset of $\mathbb{C}^N$ containing all the common zeros of the polynomials in the ideal $I$, that is:
\[
V(I):=\{ (\gamma_1, \gamma_2, \dots, \gamma_N) \in \mathbb C ^N | f(\gamma_1, \gamma_2, \dots, \gamma_N)=0 \text{ for all } f \in I\}.
\]
On the other hand, given a set $X \subseteq \mathbb C^N$, we define $I(X)$ as the set of all polynomials such that their zeros are the elements of $X$:
\begin{eqnarray*}
I(X) &:=& \left \{ f \in \mathbb C [c_1, c_2, \dots, c_N] | f(\gamma_1, \gamma_2, \dots, \gamma_N)=0 \right.\\
 & & \left. \text{ for all } (\gamma_1, \gamma_2, \dots, \gamma_N) \in X \right \}.
\end{eqnarray*}
\label{Df:VI}
\end{defn}

It is straightforward to show that $I(X)$ is an ideal.

\begin{rem}
The set $V(I)$ and the ideal $I(X)$ of Definition \ref{algebraic.sets} have a natural generalization to the projective space $\mathbb P^N(\mathbb C)$.
\end{rem}

\begin{defn} The radical of an ideal $I$ in a ring $R$ is the set
$$
\sqrt{I}:=\{ f \in R | f^m \in I \text{ for some } m \in \mathbb N\}.
$$
An ideal $I$ is said to be radical if $\sqrt{I}=I$.
\end{defn}

One of the fundamental theorems in algebraic geometry is the Nullstellensatz (see, for example \cite[Theorem 1.6]{Eisenbud}). We will make use of the following version of the Nullstellensatz.

\begin{thm}\label{nullstellensatz}(Nullstellensatz)
For every ideal $I$  of the polynomial ring $\mathbb C [c_1, c_2, \dots, c_N]$ we have
$$
I(V(I))=\sqrt{I}.
$$
\end{thm}

\begin{rem}
The Nullstellensatz naturally extends to projective spaces.
\end{rem}

\begin{defn} The dimension of a ring $R$, written $\dim(R)$, is the maximum integer $N$ such that there is a strictly ascending chain of prime ideals
$$
P_0 \subset P_1 \subset \dots \subset P_N.
$$
The dimension of an ideal $I \subset R$, written $\dim(R/I)$ is the dimension of the quotient ring $R/I$.
\end{defn}

In this paper we consider only ideals of dimension 0. An ideal $I$ in $\mathbb C [c_1, c_2, \dots, c_N]$ has dimension 0 if and only if $V(I)$ is a finite set of points.

\begin{prop}\label{solution.exists}
Let $\alpha_1, \alpha_2, \dots, \alpha_N$ be polynomials in $\mathbb C [c_1, c_2, \dots, c_N]$ such that the system $\{ \alpha_i =0\}_{i=1, \dots, N}$ does not have solutions at infinity. Then the system $\{ \alpha_i =0\}_{i=1, \dots, N}$ has at least one solution.
\end{prop}

\begin{proof}
It is know that the homogenized system $\left \{ \widehat{\alpha}_i = 0 \right \}_{i=1, \dots, N}$ in Definition \ref{solutions.at.infinity} must have a solution in $\mathbb P^N(\mathbb C)$ (see, for example, \cite[Theorem 7.2 Projective Dimension Theorem]{Hartshorne}). Because, by hypothesis, there are no solutions at infinity,  the system  $\{ \alpha_i =0\}_{i=1, \dots, N}$ must have a solution in the open chart $C_0 \neq 0$.
\end{proof}

Another fundamental theorem utilized in this paper is B\'ezout's Theorem, which has several formulations in the literature, each corresponding to a different level of generality. The version we are going to use is a slight modification of equation~(3) on~ p.~145 in~\cite{Fulton}.

\begin{thm}\label{Bezout} (B\'ezout)
Let $\widehat{\alpha}_1, \widehat{\alpha}_2, \dots, \widehat{\alpha}_N$ be homogeneous polynomials in $\mathbb C [C_0, C_1, C_2, \dots, C_N]$ respectively of degree $d_1, d_2, \dots, d_N$. Moreover, let $\widehat{I}$ be the ideal generated by $\widehat{\alpha}_1, \widehat{\alpha}_2, \dots, \widehat{\alpha}_N$ and assume that $\dim (\widehat{I}) = 0$, that is to say the variety associated to $\widehat{I}$ is a finite set of points. Then
$$
\sum_{P \in V(\widehat I)} \mult_{\widehat I}(P) = \prod_{j=1}^N d_j,
$$
where 
$$
\mult_{\widehat I}(P):=\dim_{\mathbb C} \left( \mathbb C [C_0,C_1, C_2, \dots, C_N]_P / \widehat{I}_P \right).
$$
\end{thm}

\begin{rem}\label{mult}
In Theorem \ref{Bezout} the ring $\mathbb C [C_0, C_1, C_2, \dots, C_N]_P$ is known as the localization of the ring $\mathbb C [C_0,C_1, C_2, \dots, C_N]$ at $P$ and $\widehat{I}_P$ is the ideal $I$ in $\mathbb C [C_0, C_1, C_2, \dots, C_N]_P$ generated by localized forms of the polynomials $\widehat{\alpha}_1, \widehat{\alpha}_2, \dots, \widehat{\alpha}_N$. The definition of multiplicity $\mult_{\widehat I}(P)$ is a natural generalization of multiplicity of a root of a polynomial when $N=1$. In this paper we will not make use of the formal definition of multiplicity. We will only use the fact that, if $X \subseteq \mathbb P^N (\mathbb C)$ is a finite set of points, then $\mult_{I(X)}(P)=1$ for all $P \in X$.
\end{rem}

The following statement is a consequence of B\'ezout's Theorem. For reader's sake we provide a sketch of the proof.

\begin{thm}\label{radical.theorem}
Let $\alpha_1, \alpha_2, \dots, \alpha_N$ be polynomials in $\mathbb C [c_1, c_2, \dots, c_N]$ respectively of degree $d_1, d_2, \dots, d_N$ such that the system $\{ \alpha_i = 0\}_{i =1 \dots N}$ does not have solutions at infinity. Moreover, let $\widehat{I}$ be the ideal generated by the homogenizations $\widehat{\alpha}_1, \widehat{\alpha}_2, \dots, \widehat{\alpha}_N$ of $\alpha_1, \alpha_2, \dots, \alpha_N$ and assume that $\dim (I) =\dim(\widehat{I})= 0$. If $\widehat{I}$ is a radical ideal, then the system $\{ \alpha_i = 0\}_{i =1 \dots N}$ has exactly $\displaystyle \prod_{j=1}^N d_j$ solutions.
\label{thm1}
\end{thm}

\begin{proof}
From the Nullstellensatz (Theorem \ref{nullstellensatz}), we have $\widehat{I}=I(V(\widehat{I}))$. Therefore, by applying the fact mentioned in Remark \ref{mult}, we have
$$
|V(\widehat{I})|=\sum_{P \in V(\widehat I)} 1 = \sum_{P \in V(\widehat I)} 1 = \sum_{P \in V(\widehat I)} \mult_{\widehat I} (P) = \prod_{j=1}^N d_j.
$$
Because, by hypothesis, the system $\{ \alpha_i = 0\}_{i =1 \dots N}$ does not have solutions at infinity, all the solutions must be in the affine chart $\mathbb C^N$.
\end{proof}

\section{Some Properties of the Fixed Points of the Ulam Maps}

Recall that a point ${\gamma}=(\gamma_1, \ldots, \gamma_N)$ is a fixed point of the Ulam map $\psi^{(N)}$  if and only if the zeros of the monic polynomial $p_N(x)=x^N+\sum_{m=1}^N \gamma_m x^{N-m}$ coincide with its coefficients:
\begin{equation}
p_N(x)=x^N+\sum_{m=1}^N \gamma_m x^{N-m}=\prod_{n=1}^N(x-\gamma_n)
\label{UlamPolypN}
\end{equation}
 or, equivalently, these coefficients satisfy the $N\times N$ system
\begin{eqnarray}
c_j=(-1)^j s_j^{(N)}(c_1, \ldots, c_N), \;\;j=1,2,\ldots, N,
\label{systN}
\end{eqnarray}
where the symmetric polynomials $s_j^{(N)}$ are given by~\eqref{sj}.

On the other hand, every fixed point  ${\tilde{\gamma}}=(\tilde{\gamma}_1, \ldots, \tilde{\gamma}_N, \tilde{\gamma}_{N+1})$ of the Ulam map $\psi^{(N+1)}$ satisfies the $(N+1)\times (N+1)$ system
\begin{eqnarray}
\tilde{c}_j=(-1)^j \left[ s_j^{(N)}(\tilde{c}_1, \ldots, \tilde{c}_N)+\tilde{c}_{N+1}s^{(N)}_{j-1} (\tilde{c}_1, \ldots, \tilde{c}_N)\right],  \label{systNp1}\\ \notag
j=1,2,\ldots, N, N+1.
\end{eqnarray}

The last equation in system~(\ref{systNp1}) reads
\begin{equation*}
\tilde{c}_{N+1}=(-1)^N \tilde{c}_1 \cdots \tilde{c}_N \tilde{c}_{N+1}.
\end{equation*}
Clearly, if $\tilde{c}_{N+1}=0$, system~\eqref{systNp1} reduces to system~\eqref{systN}  and it is easy to conclude the following.

\begin{prop}\label{prop.zeros} Suppose that $(\gamma_1, \ldots, \gamma_N) \in \mathbb C^N$ is a fixed point of the map $\psi^{(N)}$. Then the following are true.
\begin{itemize}
\item[(a)] For every positive integer $n$ the point $(\gamma_1, \ldots, \gamma_N, 0\ldots,0) \in \mathbb C^{N+n}$ is a fixed point of the map $\psi^{(N+n)}$.
\item[(b)]  If one of the components of the vector  $(\gamma_1, \ldots, \gamma_N) $ vanishes, then all the subsequent components vanish as well. That is, if   $\gamma_j=0$ for some $j\in \{1,2,\ldots, N-1\}$, then $\gamma_{j+1}=\gamma_{j+2}=\ldots=\gamma_N=0$.
\end{itemize}
\label{PropTrivialFixedPts}
\end{prop}

Another way to illustrate statement (a) of Proposition~\ref{PropTrivialFixedPts} is to say that if the coefficients of a monic polynomial $p_N(x)$ coincide with its zeros, then the coefficients of the monic polynomial $x^n p_N(x)$ also coincide with its zeros.

We now provide several  families of algebraic identities satisfied by the zeros of Ulam polynomials. 

\begin{prop}
Let $p_N(x)$ be an Ulam polynomial of degree $N$ with the coefficients (which are also the zeros) $\gamma=(\gamma_1, \ldots, \gamma_N)$,  see~\eqref{UlamPolypN}. Then the following identities hold for all integer $n$ such that $1\leq n \leq N$:

\begin{eqnarray}
&&\gamma_n^N+\sum_{m=1}^N \gamma_m (\gamma_n)^{N-m}=0, \label{rel1}\\
&&N (\gamma_n)^{N-1} +\sum_{m=1}^{N-1} (N-m)\gamma_m (\gamma_n)^{N-m-1}=\prod_{m=1, m\neq n}^N(\gamma_n-\gamma_m), \label{rel2}\\
&&N(N-1)(\gamma_n)^{N-2}+\sum_{m=1}^{N-2} (N-m)(N-m-1)  \gamma_m(\gamma_n)^{N-m-2}\notag\\
&&\;\;\;\;\;=2\sum_{m=1, m\neq n}^N \prod_{k=1, k\neq n,m}^N (\gamma_n-\gamma_k),\label{rel3}\\
&&(\gamma_n)^N+\sum_{m=1}^N (-1)^m s_m^{(N)}(\gamma) (\gamma_n)^{N-m}=0, \label{rel4}
\end{eqnarray}
\end{prop}
where $ s_m^{(N)}$ are the symmetric polynomials defined by~\eqref{sj}.

\begin{proof}
Relation~\eqref{rel1} is obtained by the substitution $x=\gamma_n$ into identity~\eqref{UlamPolypN}. 
Differentiation of  identity~\eqref{UlamPolypN} with respect to $x$, followed by the substitution $x=\gamma_n$, leads to relations~\eqref{rel2} and~\eqref{rel3}. Relation~\eqref{rel4} is obtained from  identity~\eqref{UlamPolypN} via the substitutions $\gamma_m=(-1)^m s_m^{(N)}(\gamma)$, see system~\eqref{systN}, and $x=\gamma_n$.
\end{proof}

To formulate systems equivalent to system~\eqref{systN}, the latter system describing the fixed points of the Ulam map $\psi^{(N)}$, we need the following Lemma.

\begin{lem}
Let $(t_1, \ldots, t_N) \in \mathbb{C}^N$ be a $N$-vector such that its components $t_n$ are all different among themselves.
If $q_N$ and $r_N$ are two monic polynomials of degree $N$ such that $q_N(t_j)=r_N(t_j)$ for all $j \in \{1, \ldots, N\}$, then  $q_N \equiv r_N$, that is, $q_N$ and $r_N$ are equal as polynomials. 
\label{polInterp}
\end{lem}
\begin{proof}
Consider the polynomial $$f_{N-1}(x)=q_N(x)-r_N(x)=\sum_{m=0}^{N-1} a_m x^{m}.$$ Note that the degree  $\deg(f_{N-1})=N-1$ because the polynomials $q_{N}$ and $r_{N}$ are both monic. Because $q_N(t_j)=r_N(t_j)$ for all $j \in \{1, \ldots, N\}$, the vector of coefficients of $f_{N-1}$, which we denote by $a=(a_0, \ldots, a_{N-1})^T$, satisfies the linear system $Ta=0$ with the $N\times N$ matrix $T$ given componentwise by $T_{jk}=(t_j)^{k-1}$. The matrix $T$ is a Vandermonde matrix, its determinant is known to be nonzero as long as the numbers $t_j$ are all different among themselves. But then $Ta=0$ implies $a=0$, which, in turn, implies $f_{N-1}\equiv 0$.
\end{proof}

\begin{thm} Let $(t_1, \ldots, t_N) \in \mathbb{C}^N$ be a $N$-vector, that is, $t_n$ are fixed complex numbers.

(1) If the numbers $t_1, \ldots, t_N$ are all different among themselves, then system~\eqref{systN} is equivalent to the system
\begin{equation}
(t_j)^N +\sum_{m=1}^N c_m (t_j)^{N-m}=\prod_{m=1}^N (t_j-c_m), \;\;1\leq j\leq N.
\label{syst1}
\end{equation}
(2) If the numbers $t_2, \ldots, t_N$ are all different among themselves, then system~\eqref{systN} is equivalent to the system
\begin{eqnarray}
&&(t_1)^N +\sum_{m=1}^N c_m (t_1)^{N-m}=\prod_{m=1}^N (t_1-c_m),\notag\\
&& N (t_j)^{N-1} +\sum_{m=1}^{N-1} (N-m) (t_j)^{N-m-1} c_m\notag\\
&&\;\;\;\;\;=\sum_{n=1}^N \prod_{m=1, m\neq n}^N(t_j-c_m),
\;\;2\leq j\leq N.
\label{syst2}
\end{eqnarray}

That is, all the fixed points of the Ulam map $\psi^{(N)}$ can be found by solving either one among systems~\eqref{syst1} or~\eqref{syst2}.

\end{thm}
\begin{proof}
Statements (1) and (2) are obtained by applying  Lemma~\ref{polInterp} to the monic polynomials $q_N(x)=x^N+\sum_{m=1}^N c_m x^{N-m}$ and $r_N(x)=\prod_{n=1}^N (x-c_n)$. 

Indeed, the vector $c=(c_1, \ldots, c_N)$ solves system~\eqref{systN} if and only if $q_N \equiv r_N$ if and only if $q_N(t_j)= r_N(t_j)$ for all $j \in \{1,2, \ldots, N\}$, if and only if $c=(c_1, \ldots, c_N)$ solves system~\eqref{syst1}.

Likewise, the vector $c=(c_1, \ldots, c_N)$ solves system~\eqref{systN} if and only if $q_N' \equiv r_N'$ and $q_N(t_1)=r_N(t_1)$, if and only if $q_N'(t_j)= r_N'(t_j)$ for all $j \in \{2, \ldots, N\}$ and $q_N(t_1)=r_N(t_1)$, if and only if $c=(c_1, \ldots, c_N)$ solves system~\eqref{syst2}.
 
\end{proof}

\section{Number of Ulam polynomials of degree $N$}

Let $U_N$ be the set of Ulam polynomials of degree $N$. In this section we derive some statements on the number $|U_N|$ for arbitrary values of $N$ and also compute $|U_N|$ for small values of $N$.

In the following, we make use of the polynomials $\alpha_j$ defined by
\begin{equation}
 \alpha_j =\alpha_j(c_1,\ldots, c_N)= s_j^{(N)}(c_1, \ldots, c_N) - (-1)^j c_j, \;\; j=1,\ldots,N.
\label{alphadf}
\end{equation}

Let $I_N$ be the ideal in $\mathbb C[c_1, \dots, c_N]$ generated by the set $\{ \alpha_i\}_{i=1,2, \dots N}$. Let $V(I_N)$ be the algebraic variety generated by the ideal $I_N$, see Definition~\ref{Df:VI};  $V(I_N)$ is the set of solutions of the system $\{ \alpha_i=0\}_{i=1,2, \dots N}$.

To use Theorem \ref{radical.theorem}, we show that $\dim(I_N)=0$.

\begin{thm}\label{lemma.dim.0}
For all $N$, the ideal $I_N$ generated by the polynomials  $\{ \alpha_i\}_{i=1,2, \dots N}$ defined by~\eqref{alphadf} has dimension zero, that is,
$$\dim(I_N)=0.$$
In other words, system (\ref{alphadf}) has only finitely many solutions.
\end{thm}

\begin{proof}
The statement has already been proved in \cite{ScalaMacia}. We provide a slightly different proof in order to point out a crucial step that we will use in Corollary \ref{cor.infinity}.

Consider the projective space $\Pro^N(\mathbb C)$. We introduce the homogeneous coordinates $\left[ C_0 : C_1 : \dots : C_N \right]$ in $\Pro^N(\mathbb C)$ such that $\mathbb C^N$ is the chart corresponding to the coordinates $\displaystyle c_i=\frac{C_i}{C_0}$ for $i=1,2, \dots , N$. Using standard results in algebraic geometry, we conclude the following. If $V(I_N)$ has a component with positive dimension, then its closure $\overline{V(I_N)}$ in $\Pro^N(\mathbb C)$ must intersect the hyperplane $C_0=0$. The algebraic set $\overline{V(I_N)}$ is defined by the equations
\begin{equation}\label{hom.eq}
 s_j^{(N)}(C_1, \ldots, C_N) - (-1)^j C_j C_0^{j-1}=0, \;\; j=1,2,\ldots,N.
\end{equation}
By substituting $C_0=0$ into system (\ref{hom.eq}), we obtain the following system:
\begin{eqnarray}
s_1^{(N)}(C_1, \ldots, C_N) &=& - C_1,\notag \\
s_j^{(N)}(C_1, \ldots, C_N) &=& 0, \;\; j=2,\ldots,N-1,\notag \\
C_1C_2 \dots C_N =0.
\label{hom.eq.0}
\end{eqnarray}
The last equation in system (\ref{hom.eq.0}) implies that at least one among the $C_i$ for $i=1 \dots N$ must equal zero. By going backward through system~\eqref{hom.eq.0}, we obtain that all $C_i$ must equal zero, hence the system has no solutions in $\Pro^N(\mathbb C)$.
\end{proof}

\begin{cor}\label{cor.infinity}
System $\alpha_j=0$ for $j=1, \dots, N$ does not have solutions at infinity, that is, the compactification $\overline{V(I_N)}$ does not intersect the hyperplane $C_0=0$ in the projective space $\Pro^N(\mathbb C)$ defined above.
\end{cor}

In conclusion, we have that the system $\alpha_j=0$ for $j=1, \dots, N$ satisfies all the hypotheses of Theorem \ref{radical.theorem} except for the radicality of the ideal $\widehat{I}_N$ generated by the homogenizations $\{ \widehat{\alpha}_1, \widehat{\alpha}_2, \dots, \widehat{\alpha}_N\}$ of the polynomials $\{ \alpha_1, \alpha_2, \dots, \alpha_N\}$.
In Subsection~\ref{subseccases} we  show that $\widehat{I}_N$ is not radical  if $N \geq 4$, see Theorem~\ref{INNotRadical}. Instead, we make use of a modified ideal $\widetilde{I}^*_N$, which we verify to be radical for all $N\leq 5$, see Subsection~\ref{subseccases}.
\vspace{2mm}

Recall that $I_N$ is the ideal generated by the polynomials $\{\alpha_j\}_{j=1}^N$ defined by~\eqref{alphadf}. We define a new ideal $\widetilde{I}_N := \left< \alpha_1, \dots, \alpha_{N-1}, \widetilde{\alpha}_N\right>$ where $\widetilde{\alpha}_N = s_{N-1}^{(N-1)}(c_1, \ldots, c_{N-1}) - (-1)^N$. In addition, we make use of the ideal  $\widetilde{I}^*_N$ defined to be the ideal generated by the \textit{homogenizations} of the polynomials $\left \{ \alpha_1, \dots, \alpha_{N-1}, \widetilde{\alpha}_N \right \}$. Note that $\tilde{I}_N^*$ is a {homogeneous} ideal in $\mathbb{C}[C_0, C_1, \ldots, C_N]$ while $\tilde{I}_N$ is an ideal in $\mathbb{C}[c_1, \ldots, c_N]$.

Since $\alpha_N = c_N \cdot \widetilde{\alpha}_N$, we have $I_N \subset \widetilde{I}_N$ and therefore $V(\widetilde{I}_N) \subset V(I_N)$. In particular, $V(\widetilde{I}_N)$ contains all the solutions of system~\eqref{systN} such that $c_N \neq 0$. However, the set $V(\widetilde{I}_N)$ may still contain some solutions with $c_N=0$. Let us also define the ideal $I_N^0:= \left< I_N, c_N \right>$.

From the inclusion-exclusion principle we have
\begin{equation}\label{incl.excl}
|U_N| = |V(I_N^0)| + |V(\widetilde{I}_N)| - |V(I_N^0) \cap V(\widetilde{I}_N)|.
\end{equation}

The following  statements shed light on the number of Ulam polynomials in $V(I_N^0)$ and $V(\tilde{I}_N)$.

\begin{prop}
For all integers $N>1$, we have the identity
$$
|V(I_N^0)| = |U_{N-1}|.
$$
That is, the number of Ulam polynomials in $V(I_N^0)$ equals the number of Ulam polynomials of degree $N-1$.
\end{prop}

\begin{proof}
Let $U_N^{0}$ denote the set of Ulam polynomials of degree $N$ that have a root $x=0$  (i.e. divisible by $x$). Consider the map $\phi:U_{N-1} \to U_N^{0}$ defined by  $\phi(P(x))=x\cdot P(x)$. It is straightforward to show that $\phi$ is a bijection.
\end{proof}

\begin{lem}\label{I.tilde.solution}
The system $\left \{ \alpha_1=0, \dots, \alpha_{N-1}=0, \widetilde{\alpha}_N=0 \right \}$ has at least one solution. That is, there is at least one Ulam polynomial in $V(\tilde{I}_N)$.
\end{lem}

\begin{proof}
From Proposition \ref{solution.exists}, it is enough to show that the system $\left \{ \alpha_1=0, \dots, \alpha_{N-1}=0, \widetilde{\alpha}_N=0 \right \}$ does not have solutions at infinity. By arguing as in the proof of Theorem \ref{lemma.dim.0}, we reduce the last system to a system similar to system~\eqref{hom.eq.0} with homogeneous coordinates, from which we determine again that $C_1=C_2= \dots=C_N=0$.
\end{proof}

Recall that $\widetilde{I}^*_N$ is the homogeneous ideal in $\mathbb{C}[C_0, C_1, \ldots, C_N]$ generated by the homogenizations of the polynomials $\left \{ \alpha_1, \dots, \alpha_{N-1}, \widetilde{\alpha}_N \right \}$.

\begin{thm}\label{rad.ideal.I.tilde}
If $\widetilde{I}^*_N$ is radical, then  $|V(\widetilde{I}_N)|= (N-1)\cdot(N-1)!$ for all integers $N>0$.
\end{thm}

\begin{proof}
Since $V(\widetilde{I}_N) \subset V(I_N)$, we know from Theorem \ref{lemma.dim.0} that $|V(\widetilde{I}_N)|$ is finite. Moreover, by arguing as in the proof of Lemma \ref{I.tilde.solution}, we know that the system $\alpha_1=0, \dots, \alpha_{N-1}=0, \widetilde{\alpha}_N=0$ does not have solutions at infinity. Therefore we can apply Theorem \ref{radical.theorem}.
\end{proof}

\begin{prop}\label{rad.ideal.I.tildeN12345}
The ideal $\widetilde{I}^*_N$ is radical for $N=1,2,3,4,5.$
\end{prop}

\begin{proof}
It is a straightforward check by using the programming environment Maple.
\end{proof}

\subsection{Calculations of $|U_N|$ for $N=1,2,3,4,5$}\label{subseccases}

\begin{enumerate}[$\bullet$]

\item $|U_1|=1$. It is straightforward to check that the only Ulam polynomial of degree 1 is $x$.\\

\item $|U_2| = 2$. In this case, we directly solve the system
\begin{eqnarray*}
c_1 + c_2 &=& -c_1,\\
c_1c_2&=&c_2,\\
\end{eqnarray*}
and obtain $U_2=\{ x^2, x^2 + x - 2\}$ and $|U_2|=2$.\\

\item $|U_3| = 6$. We directly solve system~\eqref{systN} again, this time for $N=3$, but now we refer to equation (\ref{incl.excl}). We already know that $|V(I_3^0)| = |U_2| = 2$. Moreover, the ideal $\widetilde{I}_3$ is generated by
\begin{eqnarray*}
&& 2c_1 + c_2 + c_3,\\
&& c_1c_2 + c_1c_3 + c_2c_3 - c_2,\\
&& c_1c_2 -1.\\
\end{eqnarray*}
By solving the associated system for $c_1$, we obtain that $c_1$ must be a zero of the polynomial $(x - 1)\cdot (2x^3 + 2x^2 - 1)$. Therefore, we obtain four solutions given by
\begin{eqnarray*}
(c_1, c_2, c_3) = (1,-1,-1),\\
\displaystyle (c_1, c_2, c_3) = \left( \beta_i, -\frac{1}{\beta_i} , \frac{1}{\beta_i + 1}\right),
\end{eqnarray*}
where $\beta_1, \beta_2$ and $\beta_3$ are the three distinct zeros of $2x^3 + 2x^2 - 1$.

In all of these four solutions we have $c_3\neq 0$, therefore the set $$|V(I_3^0) \cap V(\widetilde{I}_3)|$$ is empty. In conclusion, we have $|U_3| = 2 + 4 = 6$.\\

\item $|U_4| = 23$. We already know that $|V(I_4^0)| = |U_3| = 6$. We verified that the ideal $\widetilde{I}^*_4$ is radical using Maple. Therefore, from  Theorem \ref{rad.ideal.I.tilde}, we have $|V(\widetilde{I}_4)| = 18$.

In order to determine $ |V(I_4^0) \cap V(\widetilde{I}_4)|$, we need to find the number of solutions of the system
\begin{eqnarray*}
&& c_1 + c_2 + c_3 = -c_1,\\
&& c_1c_2 + c_1c_3 + c_2c_3 = c_2,\\
&& c_1c_2c_3 = - c_3,\\
&& c_1c_2c_3 = 1,\\
\end{eqnarray*}
which is the number of solutions in the case $N=3$ with the extra condition $c_1c_2c_3=1$. A simple check shows that exactly one of the solutions in the previous case satisfies this extra condition: $(c_1, c_2, c_3) = (1,-1,-1)$.

In conclusion, we have $|U_4| = 6 + 18 - 1 = 23$.\\

\item $|U_5| = 119$. We already know that $|V(I_5^0)| = |U_4| = 23$. We verified that the ideal $\widetilde{I}^*_5$ is radical using Maple. Therefore, from Theorem \ref{rad.ideal.I.tilde}, we have $|V(\widetilde{I}_5)| = 96$.

The number $ |V(I_5^0) \cap V(\widetilde{I}_5)|$ is equal to the number of solutions of the system 
\begin{eqnarray*}
&& c_1 + c_2 + c_3 + c_4 = -c_1,\\
&& c_1c_2 + c_1c_3 + c_2c_3 +c_4(c_1 + c_2 + c_3)= c_2,\\
&& c_1c_2c_3 + c_4 (c_1c_2 + c_1c_3 + c_2c_3)= - c_3,\\
&& c_1c_2c_3c_4=c_4\\
&& c_1c_2c_3c_4 = -1.\\
\end{eqnarray*}
From the last two equations we obtain that $c_4=-1$. Therefore, the system can be rewritten as follows:
\begin{eqnarray}\label{system.U5}
\begin{array}{l}
c_1 + c_2 + c_3 = 1 -c_1,\\
c_1c_2 + c_1c_3 + c_2c_3 = 1 - c_1 + c_2,\\
c_1c_2c_3 = 1 - c_1 + c_2 - c_3,\\
c_1c_2c_3 = 1.
\end{array}
\end{eqnarray}

From the last two equations we obtain $$-c_1 + c_2 - c_3 =0.$$ By combining this last equation with the first of the two equations, we obtain
\begin{eqnarray}\label{c2c3}
\begin{array}{l}
\displaystyle c_2 = \frac{1}{2}(1-c_1),\\
\\
\displaystyle c_3 = \frac{1}{2}(1 -3c_1).
\end{array}
\end{eqnarray}
We substitute the last two expressions for $c_2$ and $c_3$ into the equation $ c_1c_2 + c_1c_3 + c_2c_3 = 1 - c_1 + c_2$ to obtain a polynomial of degree two in $c_1$ with the roots
$$
c_1 = \frac{3 \pm 4i}{5}.
$$
Now, we use equations (\ref{c2c3}) in order to determine $c_2$ and $c_3$. It is straightforward to check that in  both of the two cases we have $c_1c_2c_3\neq1$. Therefore,  system (\ref{system.U5}) has no solutions and $ |V(I_5^0) \cap V(\widetilde{I}_5)| = 0$. 

In conclusion, we have $|U_5| = 23 + 96 = 119$.\\
\end{enumerate}

\begin{thm}
For all $N \geq 4$, the system $\{ \alpha_j=0\}_{j=1,2, \dots, N}$ has a solution of multiplicity larger than 1. In particular, $|U_N|< N!$ for all $N \geq 4$ and the ideal  $\widehat{I}_N$ generated by the homogenizations of polynomials $\{ \alpha_j\}_{j=1,2, \dots, N}$ is not radical for $N\geq 4$.
\label{INNotRadical}
\end{thm}

\begin{proof}
B\'ezout's Theorem tells us that there is a solution of multiplicity larger than 1 if and only if $|U_N|<N!$. From formula (\ref{incl.excl}) we have that $|U_N|=N!$ if and only if the following three conditions are simultaneously satisfied:
\begin{enumerate}[(a)]
\item $|U_{N-1}|=(N-1)!$,
\item $|V(\widetilde{I}_N)|=(N-1)\cdot(N-1)!$
\item  $|V(I_N^0) \cap V(\widetilde{I}_N)|=0$.
\end{enumerate}
If for some $N_0$ any of the above conditions is false, then $|U_N|<N!$ for all $N \geq N_0$. As we showed above, we have $|U_4|=23 < 4!$.
\end{proof}

\begin{rem} It is evident that $|U_N|\geq 1$ because $x^N$ is a (trivial) Ulam polynomial. Sharper lower bounds for $U_N$ can be obtained by using Proposition~\ref{prop.zeros}. If $p_m(x)$ is  an Ulam polynomial of degree $m$, then $x^k p_m(x)$ is an Ulam polynomial of degree $m+k$. In Subsection~\ref{subseccases} we showed the existence of $8$ nontrivial Ulam polynomials of degrees $2$ and $3$, thus there exist at least $8$ nontrivial Ulam polynomials of degree $N$, for all $N\geq 4$. 
In general, for every positive integer $N$, there exists an injective map $U_N \to U_{N+1}$ given by the multiplication by the variable $x$. In particular, $|U_{N+1}|\geq|U_{N}|$. Moreover, $|U_{N+1}|>|U_{N}|$ if and only if there exists an Ulam polynomial of degree $N+1$ whose coefficients are all different from 0, that is to say, if and only if there exists an Ulam polynomial for which $c_N \neq 0$.
\end{rem}

\begin{thm}
For every integer $N \geq 2$, there exists an Ulam polynomial of degree either $N$ or $N-1$ with all coefficients different from zero.
\end{thm}

\begin{proof} From Lemma \ref{I.tilde.solution}, the set $V(\widetilde{I}_N)$ must have at least one element $(\gamma_1, \gamma_2, \dots, \gamma_N)$. We observe that if $\gamma_k=0$ for some $k<N$, then all the subsequent $\gamma_m=0$ for all $m$ such that $k\leq m\leq N$, cf. Proposition~\ref{prop.zeros}. Therefore, if $\gamma_N \neq 0$, then  $\gamma_k\neq 0$ for all $k<N$ and we are done.

Otherwise, if $\gamma_N=0$, then a $N$-tuple $(\gamma_1, \dots, \gamma_{N-1}, 0)$ is a solution of the system $\alpha_1, \dots, \alpha_{N-1}, \widetilde{\alpha}_N$ if and only if 
$$
P(x):= x^{N-1} + \sum_{i=1}^{N-1} \gamma_i x^{N-1-i}
$$
is an Ulam polynomial of degree $N-1$. Moreover, since from equation $\widetilde{\alpha}_N$ we have $\gamma_1 \gamma_2 \dots \gamma_{N-1} \neq 0$, all the coefficients of $P(x)$ are different from zero.
\end{proof}

In particular we  obtain the following result.

\begin{cor}
For every positive integer $N$ there exists an Ulam polynomial of degree $N$ or $N+1$ with all coefficients different from zero.
\end{cor}

The reader who is interested in a different approach to studying Ulam polynomials may check \cite{Ca-Le}, where the authors mainly focus on Ulam polynomials having one coefficient equal to 0, or one coefficient equal to 1, or neither.

\section{Ulam Polynomial Eigenfunctions of Hypergeometric Type Differential Operators}

In this section we answer the following question: Are there sequences of Ulam polynomials that are eigenfunctions  of hypergeometric type differential operators? This question is related to a more general question: Are there sequences of Ulam polynomials that are orthogonal with respect to some measure?

Let $\{p_N(x)\}_{N=0}^\infty$ be a sequence of Ulam polynomials, in which the $N$-th term is given by $p_N(x)=x^N+\sum_{n=1}^N \gamma^{(N)}_n x^{N-n}=\prod_{m=1}^N(x-\gamma^{(N)}_m)$. 
Of course, for each $N \in \mathbb{N}$, the coefficients $(\gamma_1^{(N)}, \ldots, \gamma_N^{(N)})$ solve system~\eqref{systN}. A monic polynomial $y_N(x)=x^N+\sum_{j=1}^N C_j^{(N)}x^{N-j}$ of degree~$N$ is an eigenfunction of the differential operator $p(x)\frac{d^2}{dx^2}+q(x) \frac{d}{dx}$, that is
\begin{equation}
p~y_N''+q~y_N'+\lambda_N~y_N=0,
\label{DiffEqnHypergeomTypeN}
\end{equation}
where $p(x)=\alpha x^2+\beta x +\delta$ and $q(x)=-(x+a_1)$, if and only if $\lambda_N=N-N(N-1)\alpha$ and the coefficients $\{C_j^{(N)}\}_{j=1}^N$ satisfy the recurrence relations
\begin{eqnarray}
&&C_1^{(N)}=\frac{N a_1-N(N-1) \beta}{\lambda_N-\lambda_{N-1}},\notag\\
&&C_2^{(N)}=\frac{(N-1)\left[a_1-(N-2) \beta \right]}{\lambda_N-\lambda_{N-2}} C^{(N)}_1-\frac{ N(N-1) \delta}{\lambda_N-\lambda_{N-2}},\notag\\
&&C_j^{(N)}=\frac{(N-j+1)[a_1-(N-j)\beta] }{\lambda_N-\lambda_{N-j}} C_{j-1}^{(N)}\notag\\
&&-\frac{(N-j+2)(N-j+1)\delta}{\lambda_N-\lambda_{N-j}} C_{j-2}^{(N)}, \; j=3,\ldots,N,\notag\\
\label{RecRel}
\end{eqnarray}
compare with formula~(2${}'$) in~\cite{Brenke} and note the errors in that formula. Here and below we assume that the eigenvalues $\{\lambda_N\}_{N=0}^\infty$ are all different among themselves. Also, note that $C^{(1)}_1=a_1.$

Suppose that each polynomial $p_N(x)$ in the sequence $\{p_N(x)\}_{N=0}^\infty$ of Ulam polynomials solves differential equation~\eqref{DiffEqnHypergeomTypeN}. Then its coefficients $(\gamma_1^{(N)}, \ldots, \gamma_N^{(N)})$ not only solve system~\eqref{systN}, but also satisfy recurrence relations~\eqref{RecRel}. In particular, in this case the parameter $a_1$ vanishes because $p_1(x)=x+a_1$ is an Ulam polynomial if and only if $a_1=0$.

For example, if $N=2$, recurrence relations~\eqref{RecRel} with $a_1=0$ yield 
\begin{eqnarray}
\gamma_1^{(2)}=\frac{-2\beta}{1-2\alpha},\notag\\
\gamma_2^{(2)}=\frac{\delta}{\alpha-1}.
\label{eqCoefGamma2}
\end{eqnarray}
 Similarly, if $N=3$, recurrence relations~\eqref{RecRel} with $a_1=0$ yield
 \begin{eqnarray}
\gamma_1^{(3)}=\frac{-6\beta}{1-4\alpha},\notag\\ 
\gamma_2^{(3)}=\frac{3 \left[2 \beta^2+\delta(4\alpha-1)\right]}{(1-3\alpha)(1-4\alpha)},\notag\\
\gamma_3^{(3)}=\frac{4\beta \delta }{(1-2\alpha)(1-4\alpha)}.
\label{eqCoefGamma3}
 \end{eqnarray}

By substituting the above expressions for the five coefficients $\gamma_1^{(2)}, \gamma_2^{(2)}, \gamma_1^{(3)}, \gamma_2^{(3)}, \gamma_3^{(3)}$ into the system
\begin{eqnarray*}
2\gamma_1^{(2)}+\gamma_2^{(2)}=0,\\
\gamma_2^{(2)}-\gamma_1^{(2)}\gamma_2^{(2)}=0\\
2\gamma_1^{(3)}+\gamma_2^{(3)}+\gamma_3^{(3)}=0,\\
\gamma_2^{(3)}-\gamma_1^{(3)}\gamma_2^{(3)}-\gamma_2^{(3)}\gamma_3^{(3)}-\gamma_1^{(3)}\gamma_3^{(3)}=0,\\
\gamma_3^{(3)}+\gamma_1^{(3)} \gamma_2^{(3)}\gamma_3^{(3)}=0,
\end{eqnarray*}
taking into account that $\alpha \neq 1/2$ (otherwise $\lambda_1=\lambda_2$), we obtain $a_1=\beta=\delta=0$. 
Note that the first two equations of the last system ensure that $p_2(x)=x^2+\gamma_1^{(2)}x+ \gamma_2^{(2)}$ is an Ulam polynomial, while the remaining three equations ensure that $p_3(x)=x^3+\gamma_1^{(3)} x^2+ \gamma_2^{(3)}x+ \gamma_3^{(3)}$ is also an Ulam polynomial. Therefore, by~\eqref{eqCoefGamma2},\eqref{eqCoefGamma3}, $\gamma_1^{(2)}=\gamma_2^{(2)}=\gamma_1^{(3)}=\gamma_2^{(3)}=\gamma_3^{(3)}=0$.  Moreover, from~\eqref{RecRel} we conclude that $\gamma_j^{(N)}=0$ for all $j=1,\ldots, N$, where $N \in \mathbb{N}$.

It is easy to verify that for each $N \in \mathbb{N}$, the Ulam polynomial $p_N(x)=x^N$ solves  differential equation~\eqref{DiffEqnHypergeomTypeN} with  $a_1=\beta=\delta=0$. Thus we have proved the following result.

\begin{thm} If a system of polynomials $\{y_N(x)\}_{N=0}^\infty$ is such that each $y_N$ is an Ulam polynomial of degree $N$ and a solution of differential equation~\eqref{DiffEqnHypergeomTypeN}, then $a_1=\beta=\delta=0$ and $y_N=x^N$ for all $N=0,1,2\ldots$ In other words, the only Ulam polynomial eigenfunctions of hypergepometric type differential operators are the polynomials $\{x^N\}_{N=0}^\infty$, which are eigenfunctions of the differential operator $\alpha x^2 \frac{d^2}{dx^2}-x\frac{d}{dx}$ with the corresponding eigenvalues $\{N(N-1)\alpha-N\}_{N=0}^\infty$.
\end{thm}

\section{Zeros of Ulam Polynomials as Equilibria of Certain Dynamical Systems}

Let ${\gamma}=(\gamma_1, \ldots, \gamma_N)$ be the coefficients of an Ulam polynomial
\begin{equation}
p_N(z)\equiv  p_N(z, {\gamma})=z^N+\sum_{m=1}^N \gamma_m z^{N-m}=\prod_{n=1}^N (z-\gamma_n)
\end{equation}
such that the components of ${\gamma}$ are all different among themselves. Consider the  polynomial 
\begin{eqnarray}
&&q_N(z,t)\equiv q_N(z,t; a, {b})\notag\\
&&= e^t z^N+\sum_{m=1}^N \left[ \gamma_m ( e^t+a)+b_m \right] z^{N-m}
=e^t \prod_{n=1}^N (z-\zeta_n(t))
\label{df:qN}
\end{eqnarray}
with time-dependend coefficients, where $a$ is a constant and ${b}=(b_1,\ldots, b_N)$ is a $N$-vector of constants, while $\zeta_n(t)$ are the zeros of the polynomial $q_N(z,t)$. Upon differentiation of $q_N(z,t)$ defined by~\eqref{df:qN} with respect to $t$ followed by the substitution $z=\zeta_n(t)$, we obtain a system of nonlinear ODEs satisfied by the time-dependent zeros $\zeta_n(t)$ of $q_N(z,t)$:
\begin{eqnarray}
\dot{\zeta}_n(t)=-\left[ \prod_{\ell=1, \ell \neq n}^N (\zeta_n-\zeta_\ell)^{-1} \right] \left[ (\zeta_n)^N+\sum_{m=1}^N \gamma_m (\zeta_n)^{N-m}\right].
\label{zetandot}
\end{eqnarray}
System~\eqref{zetandot} is \textit{solvable} in the sense that the process of finding its solutions can be reduced to the process of finding the zeros of the polynomial $q_N(z,t)$.
Clearly, the vector of coefficients (and the zeros) ${\gamma}$ of the Ulam polynomial $p_N(z; {\gamma})$ is an equilibrium of system~\eqref{zetandot}. The same is true for each of the distinct vectors ${\gamma}_\sigma$ obtained by permuting the components of ${\gamma}$, where $1\leq \sigma\leq N!$.

Let us linearize system~\eqref{zetandot} about its equilibrium ${\gamma}$. For convenience, let us denote the right-hand side of the $n$-th equation in system~\eqref{zetandot} by $f_n({\zeta})$, where ${\zeta}=(\zeta_1,\ldots, \zeta_N)$ and consider the vector function $${f}({\zeta})=\left(f_1({\zeta}), \ldots, f_N({\zeta})\right)$$
so that system~\eqref{zetandot} is recast in the form
\begin{equation}
\frac{d {\zeta}}{dt}={f}({\zeta}).
\label{zetandotf}
\end{equation}
Note that the function $f$ is of class $C^2$ in an open neighborhood of the point $\gamma$ because the components of $\gamma$ are all different among themselves.
 By Taylor's Theorem, there exists a constant $\alpha>0$ such that for every $\zeta$ in the open ball $B({\gamma}, \alpha)$ centered at ${\gamma}$ and having the radius $\alpha$ we have
\begin{equation}
f({\zeta})= Df(\gamma) (\zeta-\gamma)+\tilde{g}(\zeta).
\end{equation}
Moreover, there exist positive constants $\beta$ and $\kappa$ such that for every $\zeta \in B(\gamma, \beta)$ we have
$$
\left| \tilde{g}(\zeta)\right|<\kappa \left| \zeta-\gamma \right|^2.
$$
It is easy to verify that $Df(\gamma) =-I$ is the negative of the $N\times N$ identity matrix $I$. We thus recast system~\eqref{zetandot} or~\eqref{zetandotf} as 
\begin{equation}
\dot{\xi}=-\xi+g(\xi),
\label{xidotNonlin}
\end{equation} 
where $\xi(t)=\zeta(t)-\gamma$ and $g(\xi)=\tilde{g}(\xi+\gamma)$ satisfies 
$$
\left| {g}(\xi)\right|<\kappa \left|\xi \right|^2
$$
for all $\xi \in B(0, \beta)$.
The fundamental matrix solution of the linearization 
\begin{equation}
\dot{y}=-y,
\label{xidotLinearized}
\end{equation} 
of system~\eqref{xidotNonlin} is $e^{-It}$, where $y(t)=\left( y_1(t), \ldots, y_N(t)\right)$. Therefore, by the variation of parameters formula~\cite{ChiconeBook}, the solution $\xi(t)$ of system~\eqref{xidotNonlin} with the initial condition $\xi(t_0)=\xi_0$ is given by
\begin{eqnarray}
\xi(t)=e^{-I(t-t_0)} \xi_0 +\int_{t_0}^t e^{-I(t-s)} g\left( \xi(s)\right)\, ds.
\end{eqnarray}
By a theorem about stability of equilibria of nonlinear dynamical systems~\cite{ChiconeBook}, $0 \in \mathbb{R}^N$ is an asymptotically stable equilibrium of  system~\eqref{xidotNonlin}, hence $\gamma$ is a stable equilibrium of system~\eqref{zetandot}.

We summarize the last result in the following theorem.

\begin{thm} If $\gamma=(\gamma_1, \ldots, \gamma_N)$ is a fixed point of the Ulam map $\psi^{(N)}$ such that its components $\gamma_n$ are all different among themselves, then $\gamma$ is an asymptotically stable equilibrium of the solvable nonlinear dynamical system~\eqref{zetandot}.
\end{thm}

\section{Discussion and Outlook}

The authors plan to improve the  results reported in this paper by obtaining sharper estimates or exact formulas for the number of Ulam polynomials of degree $N$. Other possible investigations include discovery of differential equations satisfied by Ulam polynomials, existence or non-existence of measures with respect to which sequences of Ulam polynomials are orthogonal and further investigation of dynamical systems such that their equilibria are the zeros of Ulam polynomials. We also plan to study sequences of monic polynomials  defined by
\begin{eqnarray*}
p_0(x)=x^N+\sum_{m=1}^N{c_m^{(0)} x^{N-m}},\\
p_1(x)=\prod_{m=1}^N (x-c_m^{(0)})=x^N+\sum_{m=1}^N{c_m^{(1)} x^{N-m}},\\
\ldots\\
p_n(x)=\prod_{m=1}^N (x-c_m^{(n-1)})=x^N+\sum_{m=1}^N{c_m^{(n)} x^{N-m}},\\
\ldots,
\end{eqnarray*}
in contrast with the hierarchies of monic polynomials introduced in~\cite{BC2016}. In particular, in Summer 2015 the authors conceived the idea to study periodic orbits of the operators $T(n)=p_n(x)$ defined in terms of such sequences.

\section{History and Acknowledgements}
 The work on this paper began in Summer~2015 during O.~Bihun's visit to the ``La Sapienza'' University of Rome. The first electronic draft is dated Feb. 11, 2016. Some more work on the paper was done during O.~Bihun's visit to the ``La Sapienza'' University of Rome in June~2016 and D.~Fulghesu's visit to the Scuola Normale Superiore in Pisa in June-July~2016. Both authors are grateful for the hospitality of the respective institutions. D.~Fulghesu's research is supported in part by the Simons Foundation Collaborations for Mathematicians grant \#360311.

\end{document}